\documentclass[12pt]{article}
\usepackage{amsmath}
\usepackage[english]{babel}
\usepackage{eucal}
\usepackage{stmaryrd}
\usepackage{mathabx}
\usepackage{latexsym}
\usepackage{amscd}
\usepackage{amsthm}

\newtheorem{theorem}{Theorem}
\newtheorem{lemma}[theorem]{Lemma}

\newtheorem{corollary}[theorem]{Corollary}

\theoremstyle{definition}
\newtheorem{definition}[theorem]{Definition}

\begin{document}

\begin{center}
\textbf{\large{Some properties of Neumann quasigroups}}
\end{center}

\begin{center}
\textsc{Natalia  N. Didurik and  Victor A. Shcherbacov}
\end{center}

\bigskip

\textsc{Abstract.} Any Neumann  quasigroup $(Q, \cdot)$ (quasigroup with Neumann   identity  $ x \cdot(yz \cdot yx) = z$ is called Neumann quasigroup) can be presented in the form $x\cdot y = x-y$, where $(Q, +)$ is an abelian group.  Automorphism group  of Neumann quasigroup coincides with the group $Aut(Q, +)$. Any Schweizer quasigroup  (quasigroup with  Schweizer identity $xy \cdot xz = zy$ is called Schweizer  quasigroup) is a  Neumann quasigroup and vice versa. Any Neumann quasigroup is a GA-quasigroup.

\medskip

\textbf{Keywords:} Quasigroup, Neumann quasigroup, Schweizer quasigroup,   medial quasigroup,  left Bol quasigroup, GA-quasigroup,   right nucleus

\medskip

\textbf{AMS:} 20N05

\medskip

\section{Introduction}

Basic  concepts and definitions can be found in \cite{VD, HOP, 2017_Scerb}.
Autotopies of Neumann quasigroup, groups of quasiautomorphisms and pseudoauto\-morphisms, groups of regular permutations of Neumann quasigroup are researched in \cite{DIdurik_07, DIdurik_08}. Automorphisms, quasautomorphisms, pseudoautomorphisms
of Neumann quasigroup are described in \cite{DIdurik_07, DIdurik_08}.

\begin{definition}
Algebra $(Q, \cdot, \slash, \backslash)$ with three binary operations  that satisfies  the following   identities:
\begin{equation}
x\cdot(x \backslash y) = y, \label{(1)}
\end{equation}
\begin{equation}
(y / x)\cdot x = y, \label{(2)}
\end{equation}
\begin{equation}
x\backslash (x \cdot y) = y, \label{(3)}
\end{equation}
\begin{equation}
(y \cdot x)/ x = y, \label{(4)}
\end{equation}
 is called an equational quasigroup (often  it is called  a quasigroup).
 \end{definition}

\begin{definition}
A quasigroup $(Q, \cdot)$ is said to be Neumann quasigroup    if in this quasigroup  the identity
\begin{equation}\label{Neumann}
x \cdot(yz \cdot yx) = z
\end{equation}
 holds  true \cite{Higman, AS_57, GVARAM_69}, \cite[p. 248]{STEIN}.
\end{definition}

In the articles \cite{Higman, STEIN, GVARAM_69} the following result is pointed out.

\begin{theorem}\label{13_parastrophe}
If  quasigroup $(Q, \cdot)$ satisfies the following identity
\begin{equation} \label{Second_Equation}
x y \cdot z = y\cdot zx,
\end{equation}
   then $(13)$-parastrophe  of this quasigroup satisfies Neumann   identity (\ref{Neumann}).
\end{theorem}

Notice that the identity (\ref{Second_Equation}) has also the following identity as its $(13)$-parastrophe
$(x \cdot yz) \cdot xy = z$ \cite[p. 313]{GVARAM_69},  \cite{DIdurik_08}.

\begin{definition}
Quasigroup $(Q, \cdot)$  is unipotent if and only if $x\cdot x = a$ for all $x\in Q$ and some fixed element $a\in Q$.
\end{definition}

\begin{definition}
Quasigroup $(Q, \cdot)$  has left unit element (a left unit) if there exists element $f$ (unique) such that  $f\cdot x = x$ for all $x\in Q$.
\end{definition}

\begin{definition}
Quasigroup $(Q, \cdot)$  has right unit element (a right unit) if there exists element $e$ (unique) such that  $x\cdot e = x$ for all $x\in Q$.
\end{definition}

\begin{definition}
A quasigroup $(Q, \cdot)$ is said to be left Bol quasigroup if the identity
\begin{equation}\label{Left_Bol}
x(y \cdot xz) = R^{-1}_{e_x}(x\cdot yx)\cdot z
\end{equation}
 holds true, where $x \cdot e_x = x$ for any $x\in Q$ \cite{BF_1966}.
\end{definition}

Quasigroup $(Q, \cdot)$ that satisfies the identity $xy\cdot uv = xu\cdot yv$ is called medial.
Any medial quasigrop $(Q, \cdot)$ can be be presented in the form $x\cdot y = \varphi x +\psi y + d$, where $(Q, +)$ is an abelian group, $\varphi, \psi \in Aut(Q,+)$, $\varphi \psi = \psi \varphi$, element $d$ is a fixed element of the set $Q$ \cite{VD, HOP}.

\begin{definition}
A quasigroup $(Q, \cdot)$ is said to be Moufang  quasigroup if the identity
\begin{equation}\label{Left_Bol}
x(y \cdot xz) = (x\cdot yf_x)x\cdot z
\end{equation}
 holds true, where $f_x x = x$ for any $x\in Q$ \cite{BF_1966}.
\end{definition}

Core of any loop  $(Q, \cdot)$ is defined in \cite{VD}. Core of Neumann quasigroup is defined in \cite{FLOR_BOL_QUAS_65}.

\begin{definition}
 Groupoid $(Q, \circ)$ of the form $x\circ y = x\cdot yx$, where $(Q, \cdot)$ is Neumann quasigroup, is called core of this quasigroup \cite{FLOR_BOL_QUAS_65}.
\end{definition}

\begin{definition}
The set \begin{equation*}
N_r = \{ a \in Q \mid x \cdot (y\cdot a) = (x\cdot y) \cdot a \quad \textrm {for all} \quad x, y \in Q \}
\end{equation*}
is called a right nucleus of quasigroup $(Q, \cdot)$ \cite{HOP}.
\end{definition}

\begin{definition}
A bijection $\theta$  of a set $Q$ is called a right pseudoautomorphism
of a quasigroup $(Q, \cdot)$ if there exists at least one element $c\in Q$
such that
\begin{equation} \label{Right_Psda}
\theta x \cdot  (\theta y \cdot  c) = (\theta(x \cdot y)) \cdot  c
\end{equation}
 for all $x, y \in  Q$, i.e.,  $(\theta,R_c\theta,R_c\theta)$  is an
autotopy a quasigroup $(Q, \cdot)$. The element $c$ is called a companion of $\theta$ \cite{HOP, KUNEN_99,  SCERB_03}.

A bijection $\theta$  of a set $Q$ is called a left pseudoautomorphism
of a quasigroup $(Q, \cdot)$ if there exists at least one element $c\in Q$
such that
\begin{equation} \label{Right_Psda}
(c\cdot \theta x) \cdot  \theta y  = c \cdot (\theta(x \cdot y))
\end{equation}
 for all $x, y \in  Q$, i.e.,  $(L_c \theta, \theta,L_c\theta)$  is an
autotopy a quasigroup $(Q, \cdot)$. The element $c$ is called a companion of left pseudoautomorphism  $\theta$ \cite{HOP, KUNEN_99,  SCERB_03}.
\end{definition}

Notice, in \cite{VD} right  pseudoautomorphism is called a left pseudoautomorphism and vice versa.

The set of all left (right) pseudoautomorphisms of a quasigroup $(Q, \cdot)$ forms  the group $LP(Q, \cdot)$ ($RP(Q, \cdot)$, respectively) \cite{HOP, SCERB_03}.

\begin{definition} \label{Left_GA_quasigroup}
If the group of the third components of the group $LP(Q, \cdot)$ of quasigroup $(Q, \cdot)$ acts transitively on the set $Q$, then quasigroup $(Q, \cdot)$ is called left G-quasigroup
\cite{KUNEN_99}.

If the group of the third components of the group $RP(Q, \cdot)$ of quasigroup $(Q, \cdot)$ acts transitively on the set $Q$, then quasigroup $(Q, \cdot)$ is called right G-quasigroup
\cite{KUNEN_99}.
\end{definition}

\begin{theorem} \label{identity_Element}
Any quasigroup with at least one non-trivial right pseudo\-automorphism has right identity element;

Any quasigroup with at least one non-trivial left pseudo\-automorphism has left identity element
\cite{VD, HOP} \cite[p. 23]{SCERB_03};
\end{theorem}

\section{Some results}

\begin{theorem} \label{Abelian_Group}
If quasigroup $(Q, \cdot)$ satisfies the identity (\ref{Second_Equation}), then this quasigroup is an abelian group.
\end{theorem}
\begin{proof}
We put in the identity (\ref{Second_Equation}) $x=e_z$. Then we have $e_z y \cdot z = y\cdot ze_z = yz$, $e_z y  = y$. Therefore element $e_z$ is left unity of quasigroup $(Q, \cdot)$, i.e., $e_z\cdot x = x$ for all $x\in Q$. Let be $e_z = f$.

Suppose that in the identity (\ref{Second_Equation}) $x=f$. Then $yz = y\cdot zf$ and after cancellation  $z = zf$. Therefore the element $f$ is the unity element of quasigroup $(Q, \cdot)$, i.e., quasigroup $(Q, \cdot)$ is a loop.

If we substitute in the identity (\ref{Second_Equation}) $y=f$, then we obtain that  loop $(Q, \cdot)$ is commutative.

Applying commutativity to the identity (\ref{Second_Equation}) we obtain  associative identity  $yz \cdot x  = y\cdot zx $.
\end{proof}

\begin{theorem} \label{Neyman_quas}
Any Neumann quasigroup $(Q, \cdot)$ is isotope of an abelian group $(Q, +)$ of the form $x\cdot y = x - y$.
\end{theorem}
\begin{proof}
The proof follows from Theorems \ref{13_parastrophe},   \ref{Abelian_Group} and the following fact (\cite[p. 53]{1a})
: (13)-parastrophe $(Q, \cdot)$ of IP-loop (it is obvious, that any abelian group  $(Q, +)$ is an IP-loop) is its isotope of the form $(\varepsilon, \rho, \varepsilon)$, i.e., $x\cdot y = x - y$, because in abelian groups $\rho = I$, where $x+Ix=0$ for all $x\in Q$.
\end{proof}

\begin{theorem}\label{AutotopyNeumann}
Any autotopy of Neumann quasigroup $(Q, \cdot)$ has the form:
\begin{equation} \label{Autotopy_form}
(L^+_a, L^+_{-b}, L^+_{a+b})\theta,
\end{equation}
where $L^+_a, L^+_{-b},  L^+_{a+b}$ are translations of abelian group $(Q, +)$, $\theta \in Aut(Q, +)$.
\end{theorem}
\begin{proof}
From Theorem \ref{Neyman_quas} it follows that quasigroup $(Q, \cdot)$ is isotope of commutative  group $(Q, +)$ of the form $(\varepsilon, I, \varepsilon)$, $I^2 = \varepsilon$, $I\theta = \theta I$ for any $\theta \in Aut(Q, +)$.

Any autotopy of group $(Q, +)$ has the form $(L^+_a, L^+_{b}, L^+_{a+b})\theta $, where translations and automorphism are the same as it pointed out in formulation of Theorem \ref{AutotopyNeumann} \cite{VD,  2017_Scerb}.

Autotopy groups of isotopic quasigroups are conjugate \cite{VD,  2017_Scerb}. Then we have:
\begin{equation}
(L^+_a\theta, IL^+_{b}\theta I, L^+_{a+b}\theta) = (L^+_a, L^+_{-b}, L^+_{a+b})\theta,
\end{equation}
since $IL^+_{b}\theta I = L^+_{-b}\theta$.
\end{proof}

\begin{corollary}
In Neumann quasigroups $(Q, \cdot)$ $Aut(Q, \cdot) = Aut(Q, +)$.
\end{corollary}
\begin{proof}
Automorphism is an autotopy with equal components. The rest follows from Theorem \ref{AutotopyNeumann}.
\end{proof}

\begin{corollary}
\label{AutotopyNeumann_quas}
Any autotopy of Neumann quasigroup $(Q, \cdot)$ has the form:
\begin{equation} \label{Autotopy_form_2}
(L^{\cdot}_a, L^{\cdot}_{Ib}, L^{\cdot}_{a\cdot Ib})\theta_1 ,
\end{equation}
where $L^{\cdot}_a, L^{\cdot}_{Ib}, L^{\cdot}_{a\cdot Ib}$ are translations of quasigroup  $(Q, \cdot)$, $\theta_1 = I\theta  \in Aut(Q, +) = Aut(Q, \cdot)$ \cite{DIdurik_08}.
\end{corollary}
\begin{proof}
From Theorem \ref{Neyman_quas} we have $x\cdot y =x + Iy$, $x+y = x\cdot Iy$, $L^+_x y = L^{\cdot}_x Iy $.
\end{proof}

It is clear that given above  forms of autotopy of Neumann quasigroup $(Q, \cdot)$ are not unique.
In \cite[Theorem 7]{DIdurik_08}  the following form of autotopy of Neumann quasigroup $(Q, \cdot)$ is given:
\begin{equation}
\left(L^{\cdot}_s, L^{\cdot}_{t}, L^{\cdot}_{s}  R^{-1}_t \right) \theta ,
\end{equation}
where $\theta \in Aut (Q, \cdot)$, $s, t$ are fixed elements of quasigroup $(Q, \cdot)$.

We list some properties of Neumann quasigroups proved in \cite{DIdurik_07, DIdurik_08}. Now we can use
Theorem \ref{Neyman_quas} by proving many of them.

\begin{corollary}
\begin{enumerate}
  \item Any Neumann quasigroup $(Q, \cdot)$  is unipotent and has right unit element;
  \item any loop which is isotope of Neumann quasigroup is a commutative group;
  \item any Neumann quasigroup is a medial quasigroup;
  \item any Neumann quasigroup is a left Bol quasigroup;
  \item any Neumann quasigroup is a Moufang quasigroup;
  \item core of Neumann quasigroup is a distributive groupoid;
  \item right nucleus of Neumann quasigroup $(Q, \cdot)$ consists of elements of the set $Q$ such that $a = -a$;
 \end{enumerate}
\end{corollary}
\begin{proof}
We can use  Theorem \ref{Neyman_quas}.

Case 6. Taking into account Theorem \ref{Neyman_quas} we can rewrite definition of core of Neumann quasigroup in the following form:
$x\circ y = 2\cdot x - y$, where $(Q, +)$ is an abelian group.

We check that groupoid $(Q, \circ)$ is right distributive.
We can rewrite the identity of right distributivity  $(x\circ  y) \circ z = (x\circ z) \circ (y\circ z)$ in the form:

the left part $(x\circ  y) \circ z = 2(2x - y)-z = 4 x - 2y -z $;

the right part $(x\circ z) \circ (y\circ z) = 2(2x - z) - (2 y -z) = 4x - 2z -2y + z = 4x -2y -z$.

Comparing both right  sides of the last two equalities we obtain announced result.

The proof for left distributivity is similar.

The identity $x\circ (y\circ z) = (x \circ y) \circ (x\circ z)$ has the form:

the left part $x\circ  (y \circ z)  = 2x - 2y+ z $;

the right part $(x\circ y) \circ (x\circ z) = 2(2x - y) - (2 x -z) = 4x - 2y -2x + z = 2x -2y +z$.

\smallskip

Case 7.  We  have $x\cdot (y \cdot a) = x - (y - a) = x - y + a$,  $(x\cdot y) \cdot a = x - y - a$.
\end{proof}

\begin{definition}
Quasigroup $(Q, \cdot)$ with identity
\begin{equation}\label{Sweizer}
y z \cdot yx   = x z
\end{equation}
is called Schweizer  quasigroup \cite[p. 313]{GVARAM_69}.
\end{definition}

By proving of the  next theorem we have used Prover 9 \cite{MAC_CUNE_PROV}.
\begin{theorem} \label{Swezer_neumav}
Any Schweizer quasigroup $(Q, \cdot)$ is a  Neumann quasigroup and vice versa.
\end{theorem}
\begin{proof}
$\Longrightarrow $ We use the form of Neumann quasigroup (Theorem  \ref{Neyman_quas}). Then identity (\ref{Sweizer}) takes the form:

left part $ y z \cdot yx = y-z -(y-x) = x-z$;

 right part $x \cdot z = x-z $.

  Then any Schweizer quasigroup is a Neumann quasigroup \cite{DIdurik_07}.

  $\Longleftarrow $ We prove that from  the identity (\ref{Sweizer}) follows  the identity (\ref{Neumann}).

Here  we use Schweizer identity in the form $yx\cdot yz = zx$.
  If in this  identity  we make substitution $x \rightarrow y\backslash x$,  use the identity (\ref{(1)})  and  obtain the following identity:
  \begin{equation} \label{eq 5}
  x\cdot yz = z\cdot (y\backslash x) \,\, \textrm{or, after changing of variables $x\leftrightarrow  z$ we have}\,\, x\cdot (y\backslash z) = z \cdot (y  x).
  \end{equation}
  If in identity (\ref{eq 5}) we put $x=y$, $z\rightarrow x$,  and use the identity (\ref{(1)}), the we have
  \begin{equation} \label{eq 7}
   y\cdot (y\backslash x) \overset{(\ref{(1)})}{=}  x = x \cdot (y  y).
  \end{equation}
  From the identity (\ref{eq 7}) we have that the element $e = yy$ is the right unit element in quasigroup $(Q, \cdot)$.

  If we put in the identity (\ref{Sweizer}) $z=e$, then we have
  \begin{equation} \label{eq 9}
  y\cdot yx = x.
  \end{equation}
  If in the identity (\ref{eq 9}) we substitute the left side of the  identity (\ref{Sweizer}) in the form $zx\cdot zy = yx$, the we obtain Neumann identity in the form $x \cdot (zy \cdot zx) = y$.

  Therefore, any Neumann quasigroup is a Sweizer quasigroup.
  \end{proof}

\begin{corollary}
 \label{Sweizerquas}
Any Sweizer quasigroup $(Q, \cdot)$ is isotope of an abelian group $(Q, +)$ of the form $x\cdot y = x - y$.
\end{corollary}
\begin{proof}
The proof follows from  Theorems \ref{Swezer_neumav} and \ref{Neyman_quas}.
\end{proof}

\section{G-property of Neumann   quasigroups}

We start from the following definition.

\begin{definition} \label{Pseudo_autom_Autot}
 A  bijection $\alpha$ of a set $Q$ is called a right A-pseudo\-automor\-phism  of a quasigroup $(Q, \cdot)$, if there exists a bijection $\beta$ of the set $Q$ such that the triple $(\alpha, \beta, \beta)$ is an autotopy of quasigroup $(Q, \cdot)$.

\smallskip

 A  bijection $\beta$ of a set $Q$ is called a left  A-pseudo\-automorphism  of a quasigroup $(Q, \cdot)$, if there exists a bijection $\alpha$ of the set $Q$ such that the triple $(\alpha, \beta, \alpha)$ is an autotopy of quasigroup $(Q, \cdot)$ \cite[Definition 1.159]{2017_Scerb}.
\end{definition}

We shall denote the above listed  components of autotopies  using the letter $\Pi$ with various indexes as follows: $_1\Pi^A_l$, $_2\Pi^A_l$, $_3\Pi^A_l$,  $_1\Pi^A_r$, $_2\Pi^A_r$, and $_3\Pi^A_r$. The letter $A$ in the right upper corner means that this is an autotopical pseudo\-automorphism. For example, $_2\Pi^A_r$ denotes the group of  second components of  right A-pseudo\-automorphisms of a quasigroup $(Q, \cdot)$.

Notice, in loop case concept of right   A-pseudo\-automorphism from Definition \ref{Pseudo_autom_Autot} is transformed into the concept of   left pseudoautomorphism in sense of the book of V.~D. Belousov  \cite[p. 45]{VD} and it is transformed into the concept of right pseudoautomorphism in terminology of H.~O. Pflugfelder \cite{HOP}.
For quasigroups that have left (right) identity element the similar result is proved in   \cite[Theorem 2]{DIdurik_08}.

\begin{definition}
 A quasigroup $(Q,\cdot)$ is called a right  GA-quasigroup, if the group $_2\Pi^A_r$ (or the group $_3\Pi^A_r$) is transitive on the set $Q$.

\smallskip

 A quasigroup $(Q,\cdot)$ is called a left GA-quasigroup, if the group $_1\Pi^A_l$ (or the group  $_3\Pi^A_l$) is transitive on the set $Q$.

\smallskip

 A right and left  GA-quasigroup is called GA-quasigroup.
\end{definition}

\begin{lemma} \label{G_A_QuasneuNeumann}
Case 1. Autotopy of Neumann  quasigroup $(Q, \cdot)$ is a right  A-pseudo\-automorphism if and only if
\begin{equation} \label{right_Pseudo}
a = -2b
\end{equation}
for some  fixed $a, b \in Q$.

Case 2. Autotopy of Neumann  quasigroup $(Q, \cdot)$  is a left   A-pseudo\-automorphism if and only if
the following equality is true:
\begin{equation} \label{Vid_Of_AutotopyNeu}
b=0
\end{equation}
for   fixed $b \in Q$.
\end{lemma}
\begin{proof}
The proof follows from   Theorem  \ref{AutotopyNeumann}.
\end{proof}

\begin{theorem} \label{Prop_223}
Any Neumann  quasigroup $(Q, \cdot)$ is a GA-quasigroup.
\end{theorem}
\begin{proof}
From Theorem \ref{AutotopyNeumann} and  Lemma \ref{G_A_QuasneuNeumann} we have:
any right  A-pseudo\-automorphism of Neumann quasigroup $(Q, \cdot)$  has the form
\begin{equation} \label{Autotopy_formpseudorigt}
(L^+_{-2b}, L^+_{-b}, L^+_{-b})\theta
\end{equation}
for all elements $b\in Q$, where $L^+_{-2b},  L^+_{-b}$ are translations of the group $(Q, +)$;

any left  A-pseudo\-automorphism has the form
\begin{equation} \label{Autotopy_formpseudoleft}
(L^+_a, \varepsilon, L^+_{a})\theta
\end{equation}
for all elements $a\in Q$, where $L^+_{a}$ is a translation of the group $(Q, +)$.

From equalities (\ref{Autotopy_formpseudorigt}) and (\ref{Autotopy_formpseudoleft}) it follows that the groups
$_3\Pi^A_r$ and $_3\Pi^A_l$ act transitively on the set $Q$. The last means that any Neumann  quasigroup $(Q, \cdot)$ is a GA-quasigroup.
\end{proof}

Notice, Neumann quasigroup is left   G-quasigroup \cite[Theorem 3]{DIdurik_08} in sense of definition of V.~D. Belousov \cite{VD}.
From the last result it follows   that Neumann quasigroup is  GA-quasigroup which is not right G-quasigroup in sense of V.~D. Belousov.

\noindent {{\sc N. N. Didurik} \\
{State University \lq\lq Dimitrie Cantemir\rq\rq }    \\
{Academiei str. 3/2, MD--2028 Chi\c sin\u au} \\
{Moldova}  \\
E-mail:  \emph{natnikkr83@mail.ru}        

 \noindent {\sc V. A. Shcherbacov}\\
  Institute of Matematics and Computer Science\\
   {Academiei str. 5, MD--2028 Chi\c sin\u au}\\
  {Moldova}\\
  E-mail:   \emph{victor.scerbacov@math.md}          

}

\end{document}